\documentclass[11pt, a4paper]{amsart}

\usepackage{amsmath, amssymb, amsthm}
\usepackage{hyperref}
\usepackage[english]{babel}
\usepackage[latin1]{inputenc}
\usepackage{xy}
\usepackage{graphicx}
\usepackage{verbatim}
\usepackage{enumerate}
\usepackage{url}
\usepackage{subfig}
\usepackage{float}
\usepackage{xcolor}

\newcommand{\maxdim}{n}

\newcommand{\ud}{\mathrm{d}}

\newcommand{\Om}{\Omega}

\newcommand{\e}{\varepsilon}
\newcommand{\vphi}{\varphi}

\newcommand{\F}{\mathcal{Y}}

\newcommand{\sm}{\smallsetminus}

\newcommand{\R}{\mathbb{R}}
\newcommand{\Rn}{\R^\maxdim}
\newcommand{\Rm}{\R^m}

\newcommand{\Z}{\mathbb{Z}}
\newcommand{\N}{\mathbb{N}}
\newcommand{\CIR}{\mathcal{C}}
\newcommand{\C}{\mathbb{C}}
\newcommand{\Cn}{\C^\maxdim}

\newcommand{\Csn}{\C_\bullet^\maxdim}
\newcommand{\Zn}{\Z_\maxdim}

\DeclareMathOperator{\Imp}{Im}

\renewcommand{\Im}{\Imp}

\newcommand{\Bm}{\beta}

\newcommand{\gh}{\hat{g}}
\newcommand{\gt}{\tilde{g}}

\newcommand{\Tor}{\mathbf{T}}
\newcommand{\Torn}{\Tor^\maxdim}
\newcommand{\Tornn}{\Tor^{2\maxdim}}

\newcommand{\X}{\mathbf{X}}

\newcommand{\g}{\gamma}
\newcommand{\G}{\Gamma}

\newcommand{\Ob}{\mathcal{O}}

\newcommand{\Sg}{\Sigma}
\newcommand{\Ss}{\mathbf{S}}
\newcommand{\Sk}[1]{\Ss_{#1}}
\newcommand{\Sn}{\Ss_\maxdim}

\newcommand{\Ai}{A_{\textrm{inc}}}
\newcommand{\Ac}{A_{\textrm{coh}}}

\newcommand{\abs}[1]{\left|#1\right|}
\newcommand{\tabs}[1]{\big|#1\big|}

\newcommand{\rset}[2]{\left\lbrace\, #1\,\left|\;#2\right.\right\rbrace}

\newcommand{\set}[2]{\rset{#1}{#2}}

\newcommand{\sset}[1]{\left\lbrace #1\right\rbrace}

\DeclareMathOperator{\id}{id}
\newcommand{\TG}{\sset{\id}}

\newcommand{\Cinf}{C^\infty}

\newcommand{\fleft}{\!\left}

\author{Christian Bick}

\address{Centre for Systems, Dynamics and Control and Department of Mathematics, University of Exeter, Exeter EX4 4QF, UK}

\title[Chimeras and Symmetry]{
Isotropy of Angular Frequencies and Weak Chimeras With Broken Symmetry
}
\date{\today}

\newtheorem{prop}{Proposition}
\newtheorem{lem}{Lemma}
\newtheorem{thm}{Theorem}
\newtheorem{cor}{Corollary}
\theoremstyle{definition}
\newtheorem{defn}{Definition}

\theoremstyle{remark}
\newtheorem{rem}{Remark}

\newcommand{\imagescaling}{0.9}

\begin{document}

\begin{abstract}
The notion of a weak chimeras provides a tractable definition for chimera states in networks of finitely many phase oscillators. Here we generalize the definition of a weak chimera to a more general class of equivariant dynamical systems by characterizing solutions in terms of the isotropy of their angular frequency vector---for coupled phase oscillators the angular frequency vector is given by the average of the vector field along a trajectory. Symmetries of solutions automatically imply angular frequency synchronization. We show that the presence of such symmetries is not necessary by giving a result for the existence of weak chimeras without instantaneous or setwise symmetries for coupled phase oscillators. Moreover, we construct a coupling function that gives rise to chaotic weak chimeras without symmetry in weakly coupled populations of phase oscillators with generalized coupling.
\end{abstract}
\maketitle


\section{Introduction}
The emergence of collective dynamics in networks of coupled oscillatory units is a fascinating phenomenon observed in science and technology~\cite{Strogatz2000, Pikovsky2003}. Symmetric phase oscillator networks provide paradigmatic models to understand collective dynamics in the weak coupling limit~\cite{Strogatz2004, Acebron2005, Tchistiakov1996, Ashwin2015}. Such dynamical systems with symmetry are equivariant with respect to the action of a group~\cite{Golubitsky1988, Golubitsky2002, Field2007a}, that is, the vector field commutes with the group action on phase space. Equivariance implies that any solution of the system is mapped to another solution by the action of the symmetry group and it typically constrains the dynamics, for example by giving rise to dynamically invariant subspaces. The solutions themselves may (but do not have to) have nontrivial symmetry, that is, there may be nontrivial elements elements of the symmetry group that keep the solution fixed, either pointwise or as a set. For example, for globally coupled identical oscillators the solution corresponding to full synchrony, where the states of all oscillators are equal, has full symmetry itself. Of course, there may be other solution with less symmetry relative to the symmetries of the system.

Recently, the observation of ``symmetry breaking'' in symmetrically coupled phase oscillator systems, i.e., the observation of solutions with localized synchronous dynamics coexisting with localized incoherence, has sparked a lot of interest. Such solutions, commonly known as chimera states---see~\cite{Panaggio2015} for a recent review---were first observed in symmetric rings of coupled phase oscillators~\cite{Kuramoto2002, Abrams2004}. In the limit of infinitely many oscillators, they correspond to stationary or periodic patterns of the phase density distribution~\cite{Abrams2008, Omel'chenko2013}. By contrast, it was not until recently that Ashwin and Burylko~\cite{Ashwin2014a} gave a testable mathematical definition for chimera states, a \emph{weak chimera}, for networks of finitely many phase oscillators whose phases $\vphi_k\in \Tor=\R/2\pi \Z$, $k=1,\dotsc,\maxdim$, evolve according to
\begin{equation}\label{eq:COsc}
\frac{\ud\vphi_k}{\ud t} = \dot\vphi_k = \omega + \frac{1}{n} \sum_{j=1}^n H_{kj} g(\vphi_k-\vphi_j).
\end{equation}
Here the~$H_{kj}$ determine the network topology (respecting a subgroup~$\G$ of the group~$\Sn$ of permutations of~$\maxdim$ symbols acting transitively on the indices of the oscillators) and~$g:\Tor\to\R$ is the generalized coupling (or phase interaction) function. Weak chimeras are defined in terms of partial angular frequency synchronization on trajectories. More precisely, if $\hat\vphi$ is a continuous lift of a solution~$\vphi$ of~\eqref{eq:COsc} with initial condition $\vphi^0$ to~$\Rn$ define the asymptotic angular frequency of oscillator~$k$ as
\begin{equation}\label{eq:Omegak}
\Omega_k(\vphi^0) = \lim_{T\to\infty}\frac{\hat\vphi(T)}{T}.
\end{equation}
According to~\cite{Ashwin2014a}, a compact, connected, chain-recurrent, and dynamically invariant set~$A\subset\Torn$ is a weak chimera if there are distinct oscillators $j,k,\ell$ such that $\Omega_j(\vphi^0)=\Omega_k(\vphi^0)\neq\Omega_\ell(\vphi^0)$ for all $\vphi^0\in A$.

Weak chimeras and angular frequency synchronization relate to symmetry. Assuming that all limits~\eqref{eq:Omegak} exist, we have a frequency vector
\[\Om(\vphi^0)=(\Om_1(\vphi^0), \dotsc, \Om_\maxdim(\vphi^0))\in\Rn.\]
The group~$\Sn$ also acts on~$\Rn$ by permuting indices. If~$A$ is a weak chimera as above and $\tau_{kj}\in\Sn$ denotes the transposition swapping indices~$k$ and~$j$ then $\tau_{kj}\Om(\vphi^0)=\Om(\vphi^0)$. That is, $\tau_{kj}$ is a symmetry of the angular frequency vector~$\Om(\vphi^0)$.
While weak chimeras have provided a suitable framework to derive for example existence results~\cite{Bick2015c}, there are two shortcomings. First, while chimera states have also been reported in more general oscillator models~\cite{Sethia2013, Zakharova2014} the definition above applies to phase oscillators only. Second, the symmetries of the angular frequency vector~$\Om$ may be different from the symmetries of the system. As a consequence, if~$A$ is a weak chimera then~$\tau_{kj}A$ may not be a weak chimera or even a solution of the system at all. Interestingly, while it has been argued that chimera states are relevant due to their nature of solutions with broken symmetry~\cite{Abrams2004}, their properties have never been phrased in terms of symmetries of the dynamical system.

The contribution of this paper is twofold: first, we give a definition of a weak chimera in the language of equivariant dynamical systems and, second, we show that symmetries of the solution are not necessary for the occurrence of weak chimeras. More precisely, we define weak chimeras in terms of the isotropy of the angular frequency vector which can be stated for more general oscillator systems. We observe that, in a suitable setup, asymptotic angular frequencies are averages of equivariant observables. Therefore, symmetries of solutions translate directly into symmetries of the angular frequencies. Thus, the presence of symmetries of solutions facilitates the emergence of weak chimeras and, in fact, most weak chimeras that have been constructed explicitly~\cite{Ashwin2014a, Panaggio2015b, Bick2015c} are solutions with (instantaneous) symmetries. Is it possible to construct weak chimeras without instantaneous or setwise symmetries for which the angular frequencies have symmetries that are not a property of the solution itself? This question motivates the second contribution. Extending recent persistence results~\cite{Bick2015c} that rely on constructing generalized coupling functions between oscillators, we prove a persistence result for weak chimeras with trivial symmetry in weakly coupled populations of phase oscillators. Moreover, we present an explicit example of a $C^\infty$ coupling function that gives rise to a chaotic weak chimera without instantaneous or setwise symmetries in a nontrivially coupled system.

This paper is organized as follows. In Section~\ref{sec:Prelims} we review some terminology on equivariant dynamics that is needed in the subsequent sections. In Section~\ref{sec:WeakChimeras} we then apply these notions to general oscillator systems with symmetry which yields a new definition of a weak chimera in terms of symmetries of the angular frequency vector. As we show in Section~\ref{sec:WeakChimerasPhaseOsc} this definition is compatible with previous definitions. In Section~\ref{sec:WeakChimNoSym} we prove a persistence result for weak chimeras without instantaneous or average symmetries. Finally, we present an explicit example of a coupling function which gives rise to chaotic weak chimeras with trivial symmetries in Section~\ref{sec:Example} and finish with some concluding remarks.

\section{Preliminaries}
\label{sec:Prelims}

\subsection{Quasi-regular points}

Let~$\X$ be a compact differentiable manifold with a flow $\Phi_t:\X\to\X$, $t\in\R$. 
A point $x\in\X$ is \emph{quasi-regular} if the limit
\[\lim_{T\to\infty}\frac{1}{T}\int_{0}^{T}f(\Phi_t(x))\ud t\]
exists for all continuous functions $f:\X\to\R$. 

\begin{thm}[\cite{Schwartzman1957, Oxtoby1952}]
The set of points which are not quasi-regular has zero measure with respect to every finite measure on~$\X$ that is invariant under the flow~$\Phi_t$.
\end{thm}

\subsection{Equivariant dynamical systems}

Let $F:\X\to \textrm{T}\X$ be a smooth vector field on~$\X$ where~$\textrm{T}\X$ denotes the tangent 
bundle. Suppose that a group~$\Gamma$ acts on~$\X$. The vector field~$F$ 
is $\G$-equivariant if
\begin{equation}
F(\gamma x) = \hat{\gamma}F(x)
\end{equation}
for all $\gamma\in\Gamma$ where~$\hat{\gamma}$ is the induced 
action on the tangent space. A $\G$-equivariant vector field defines 
a \emph{$\Gamma$-equivariant dynamical system}
\begin{equation}
\label{eq:EqDynSyst}
\dot x = F(x)
\end{equation}
on~$\X$~\cite{Golubitsky2002, Field2007a}. For a set $A\subset \X$ 
define the set of \emph{instantaneous symmetries}
\begin{align}
T(A)&=\set{\gamma\in\G}{\gamma x = x \text{ for all }x\in A}
\intertext{and the set of \emph{symmetries on average} (or setwise symmetries)}
\Sg(A)&=\set{\gamma\in\G}{\gamma A = A}.
\end{align}
Clearly, $T(A)\subset\Sg(A)$. If $\G_x = \set{\g\in\G}{\g x = x}$ 
denotes the \emph{stabilizer} or \emph{isotropy subgroup} of $x\in\X$ 
we have $T(A)=\bigcap_{x\in A}\G_x$. 

Note that if $\g A\cap A=\emptyset$ for all $\g\in\G\sm\TG$ then 
$\Sg(A)=\TG$. The converse holds only under additional 
assumptions~\cite{Ashwin1995}. Henceforth, let~$\Phi_t:\X\to\X$, $t\in\R$, denote 
the flow defined by the differential equation~\eqref{eq:EqDynSyst}. A set~$A\subset\X$ is (forward) 
\emph{flow-invariant} or \emph{dynamically invariant} if 
$\Phi_t(A)\subset A$ for all $t\geq 0$. Moreover, $A$ is \emph{stable} if for
every neighborhood~$U$ of~$A$ there exists an open 
neighborhood~$V\subset U$ of~$A$ such that $\Phi_t(V)\subset U$ for 
all $t\geq 0$. A compact stable set~$A$ is an \emph{attractor} if~$A=\omega(x)$ 
is the $\omega$-limit set of some point $x\in\X$.

For attractors and the action of the orthogonal group~$O(\maxdim)$ 
on~$\Rn$ there is the following 
dichotomy~\cite{Melbourne1993} that characterizes the symmetries 
on average.

\begin{prop}
\label{prop:Symmetry}
Let $\Gamma\subset O(n)$ be a finite subgroup. For an attractor~$A\subset\Rn$ we 
have for any $\g\in\G$ either $\g A = A$ or $\g A\cap A=\emptyset$.
\end{prop}

\begin{rem}
The same statement holds for repellers---dynamically invariant sets that are attractors when time is reversed. However, it does not necessarily hold for dynamically invariant sets of saddle type, sets that are attracting (or repelling) in a more general sense, or heteroclinic attractors.
\end{rem}

For $\delta>0$ let~$B_\delta(A)$ denote an (open) $\delta$-neighborhood 
of~$A$.

\begin{cor}
\label{cor:NoSymPerturb}
Let $\Gamma\subset O(n)$ be a finite subgroup and let~$A\subset\Rn$ be compact attractor for the flow defined by~\eqref{eq:EqDynSyst}. If $\Sg(A) = \TG$ then there exists a $\delta>0$ such that $\Sg(D)=\TG$ for any $D\subset B_\delta(A)$.
\end{cor}

\begin{proof}
Suppose that $\Sg(A) = \TG$. By Proposition~\ref{prop:Symmetry} 
we have $\g A \cap A=\emptyset$ for any $\g\neq\id$. Since~$A$ is 
closed there exists a $\delta>0$ such that
$\g B_\delta(A) \cap B_\delta(A)=\emptyset$. Therefore, $\Sg(D)=\TG$
for any $D\subset B_\delta(A)$.
\end{proof}

\subsection{Equivariant observables}

Suppose that~$\G$ acts on both~$\X$ and~$\Rm$ for some $m\in\N\sm\sset{0}$.

\begin{defn}
A continuous $\G$-equivariant map $\Ob:\X\to\Rm$ is an \emph{observable}.
\end{defn}

Given a solution~$x(t)$ of~\eqref{eq:EqDynSyst} with initial condition $x(0)=x^0$, the limit
\begin{equation}
\label{eq:AvLimit}
K_\Ob(x^0) = \lim_{T\to\infty}\frac{1}{T}\int_{0}^{T}\Ob(x(t))\ud t
\end{equation}
(if it exists) is an \emph{average of~$\Ob$ along the trajectory~$x$} (integrate componentwise if $m>1$). The limit exists in particular for every quasi-regular initial condition~$x^0$ and henceforth we will always assume that $x^0\in\X$ is quasi-regular when averages~\eqref{eq:AvLimit} are evaluated.

Suppose that~$A\subset\X$ is dynamically invariant and supports a $\Phi_t$-invariant ergodic probability measure~$\mu$. Write
\begin{equation}
\label{eq:SpatAvg}
K^\mu_\Ob(A) = \int_A\Ob(x)\ud\mu.
\end{equation}
By the Birkhoff ergodic theorem~\cite[Theorem~4.1.2]{Katok1995} we have 
\begin{equation}
\label{eq:ErgAvg}
K_\Ob(x^0) = K^\mu_\Ob(A).
\end{equation}
for~$\mu$-almost every $x^0\in A$. In particular, the limit~\eqref{eq:AvLimit} exists for $\mu$-almost every $x^0\in A$. For ease of notation, we will simply write $K_\Ob(A)=K^\mu_\Ob(A)$ unless the choice of measure is important. Of course, not every ergodic invariant measure is ``physically relevant'' since~$\mu$ may be singular with respect to the Lebesgue measure. If an attractor~$A$ supports a Sinai--Ruelle--Bowen (SRB) measure~$\mu$~\cite{Young2002, Katok1995} there is a neighborhood~$W$ of~$A$ such that~\eqref{eq:ErgAvg} holds for Lebesgue-almost every $x^0\in W$. Thus the average~\eqref{eq:SpatAvg} is observed for ``typical'' initial conditions with respect to the Lebesgue measure.

Now~$K_\Ob(A)$ has an isotropy group~$\Gamma_{K_\Ob(A)}$ and a simple calculation~\cite{Golubitsky2002} shows that
\begin{equation}
\label{eq:IsotropySubset}
\Sg(A)\subset\Gamma_{K_\Ob(A)},
\end{equation}
that is, any symmetry on average is contained in the isotropy group of the observation.

The converse does not hold for general observables.
\emph{Detectives}~\cite{Golubitsky2002, Barany1993, Dellnitz1994, Ashwin1997}
are an important class of observables for which the isotropy is generically equal to the symmetries on average. Given a suitably large~$m\in\N\sm\sset{0}$, an observable is an (ergodic) detective 
if for any $\omega$-limit set~$A$ there exists an open dense set 
of near-identity
$\G$-equivariant diffeomorphisms $\psi:\Rm\to\Rm$ such that 
$\Gamma_{K_\Ob(\psi(A))}=\Sg(A)$.
Hence, detectives are particular observables to ``detect'' the 
symmetries of attractors.


\section{Weak Chimeras and Symmetries on Average}
\label{sec:WeakChimeras}

\subsection{Isotropy of angular frequencies and weak chimeras}\label{WeakChimSym}
The symmetry point of view now allows to define weak chimeras in terms of their symmetries as solutions relative to the symmetries of the system itself. 
Write $i=\sqrt{-1}$. Let $\Gamma\subset\Sn$ be a subgroup that acts transitively on~$\Csn := (\C\sm\sset{0})^\maxdim$ by permuting coordinates and suppose that $F:\Csn\to\Cn$ is $\Gamma$-equivariant. The map~$F=(F_1, \dotsc, F_\maxdim)$ determines a dynamical system on~$\Csn$ where the evolution of~$z = (z_1, \dotsc, z_k)$\footnote{One can think of each complex variables~$z_k$ representing phase and amplitude of an oscillator.} is given by
\begin{equation}\label{eq:CmplxDyn}
\dot z_k = z_k F_k(z).
\end{equation}
For $k\in\sset{1, \dotsc, \maxdim}$ define
\begin{equation}
C_k(z) = \frac{z_k}{\abs{z_k}}.
\end{equation}
In the following we assume that~$F$ is such that (a) the dynamics of~\eqref{eq:CmplxDyn} are well defined on~$\Csn$, (b) we have~$\omega(z)\subset\Csn$ for all $z\in\Csn$ and (c)~the derivative $C_k'(z(t)) := \frac{\ud}{\ud t}C_k(z(t))$ exists for any trajectory~$z(t)$. These assumptions are easy to work with but can be relaxed as one typically only needs well-defined dynamics on a neighborhood of $\Torn\subset\Csn$.

Note that~$C_k(z_1, \dotsc, z_n)$ projects onto the unit circle in the $k$th coordinate. Let $\gamma_T$ denote the parametrized curve in~$\Csn$ determined by a solution $z(t)$ of~\eqref{eq:CmplxDyn} for $t\in[0, T]$. The change of argument of~$z_k$ along~$\gamma_T$ is given by
\begin{equation}
\Delta\arg C_k(\gamma_T) =  
\frac{1}{i}\int_0^T \frac{C_k'(z(t))}{C_k(z(t))}\ud t = \int_0^T \Im(F_k(z(t)))\ud t.
\end{equation}
Thus, we obtain the \emph{average angular frequency in the $k$th coordinate} (equivalent to the average winding number when multiplied by~$2\pi$)
\begin{equation}
\Omega_k(z^0) := \lim_{T\to\infty}\frac{1}{T}\int_0^T \Im(F_k(z(t)))\ud t
\end{equation}
along a trajectory~$z(t)$ with initial condition~$z^0$.

\begin{defn}
\label{def:FrequencyVector}
The vector \[\Om(z^0) = (\Omega_1(z^0), \dotsc, \Omega_\maxdim(z^0))\] is the \emph{angular frequency vector} of the trajectory with initial condition~$z^0\in\Csn$.
\end{defn}

Since~$\Sn$ also acts on~$\Rn$ by permuting indices,~$\Im(F):\Csn\to\Rn$ is a $\Gamma$-equivariant observable for~\eqref{eq:CmplxDyn} and 
\[\Om(z^0) = K_{\Im(F)}(z^0),\] 
that is, the angular frequency vector is the observation of~$\Im(F)$ along a trajectory. For a compact and invariant set $A\subset\Csn$ with an unique ergodic invariant measure we write $\Om(A) = K_{\Im(F)}(A)$ for the \emph{angular frequency vector of~$A$}.

The symmetries of the system~\eqref{eq:CmplxDyn} now allow to phrase angular frequency synchronization in terms of the isotropy of the angular frequency vector. An observation of $\Im(F)$ has isotropy subgroup $\G_{\Om(A)}\subset\G$. This motivates a definition of a weak chimera~\cite{Ashwin2014a}---originally limited to networks of phase oscillator---to more general oscillator systems~\eqref{eq:CmplxDyn}.

\begin{defn}
\label{def:SymWeakChimera}
A compact, connected, chain-recurrent, and dynamically invariant set~$A\subset\Csn$ is a \emph{weak chimera} for~\eqref{eq:CmplxDyn} if
\[\TG\subsetneq \Gamma_{\Om(\vphi^0)} \subsetneq \G\]
for all $\vphi^0\in A$. If a weak chimera~$A$ supports an SRB measure then it is called \emph{observable} and we have
\[\TG\subsetneq \Gamma_{\Om(A)} \subsetneq \G.\]
\end{defn}

\begin{rem}
Asymptotic winding (or rotation) numbers can be defined in a more general setting: they quantify how trajectories of a given flow wind around a topological space~$\X$; cf.~\cite{Schwartzman1957, Walsh1995} for details. These winding numbers are defined for continuous maps $f:\X\to S^1=\set{z\in\C}{\abs{z}=1}$. For spaces with finitely generated homology, it suffices to evaluate winding numbers for maps~$f_k$ corresponding to a basis of the first cohomology~\cite{Schwartzman1957}.

Here we have $\X=\Csn$ and the maps~$C_k$ defined above correspond to the generators of the homology of~$\Csn$. Since we consider flows given by a $\G$-equivariant differential equation, we characterize weak chimeras by the symmetry properties of the asymptotic winding numbers. In the language of asymptotic cycles, these are solutions where for certain ``directions'' the winding behavior is the same while for other directions it is distinct. This suggests that the notion can be further extended to equivariant dynamical systems on more general~$\X$ with nontrivial homology.
\end{rem}

Note that the weak chimeras of Definition~\ref{def:SymWeakChimera} are defined solely in terms of the symmetry properties of the system. Moreover, the definition extends beyond the weak coupling limit of interacting limit cycle oscillators~\cite{Ashwin1992}: systems of the form~\eqref{eq:CmplxDyn} describe dynamical systems close to a Hopf bifurcation~\cite{Ashwin2015a} or more general oscillator models where ``amplitude-mediated chimeras'' have been observed~\cite{Sethia2013}. Moreover, the next proposition asserts that the change of argument of~$z_k$ along $C_\ell(\gamma_T)$ cannot be bounded to obtain nontrivial winding numbers; such dynamics are observed for ``pure amplitude chimeras''~\cite{Zakharova2014} and thus our definition is sufficiently general to provide a rigorous framework for such chimeras.

\begin{prop}
Suppose that $z(t)$ is a solution of~\eqref{eq:CmplxDyn} such that there are 
$j\neq\ell$, $M,R>0$ such that $\Delta\arg C_\ell(\gamma_T)<M$ and $\Delta\arg C_j(\gamma_T)>RT$ 
for all~$T$. Then $\Gamma_{\Om(A)}\subsetneq\G$.
\end{prop}

\begin{proof}
Immediate from $\Omega_k(z^0) = \lim_{T\to \infty}\frac{1}{T}\Delta\arg C_k(\gamma_T)$.
\end{proof}

Definition~\ref{def:SymWeakChimera} is compatible with the action of the symmetry group on~$\Csn$.

\begin{prop}
If $A\subset\Torn$ is a weak chimera, so is $\g A$ for any $\g\in\G$.
\end{prop}

\begin{proof}
The assertion follows directly from $\G$-equivariance of~$F$.
\end{proof}

This implies in particular that the isotropy of the angular frequency vectors are conjugate if weak chimeras are related by symmetry. If $\g\in\Sg(A)$ then $\Om(\g A)=\Om(A)$, that is, the angular frequency vectors (and therefore the isotropy) are identical.

\subsection{Symmetries imply frequency synchronization}
\label{sec:SymImplSync}
Intuitively speaking, a weak chimera~$A$ consists of solutions of~\eqref{eq:CmplxDyn} along which the average angular frequencies have some symmetries but not too many. Inclusion~\eqref{eq:IsotropySubset} implies
\begin{equation}
\label{eq:Inclusion}
\TG\subset T(A)\subset\Sg(A)\subset\G_{\Om(A)}\subset\Gamma\subset \Sn.
\end{equation}
Consequently, if a solution has nontrivial instantaneous symmetry then the corresponding angular frequency vector has nontrivial isotropy. Similarly, the angular frequency vector of dynamically invariant sets with nontrivial setwise symmetry has nontrivial isotropy. For invariant sets with nontrivial (setwise or instantaneous) symmetry,~\eqref{eq:Inclusion} implies that one condition of Definition~\ref{def:SymWeakChimera} is automatically satisfied. In that sense the presence of symmetries ``facilitates'' the occurrence of weak chimera states. 

More generally speaking, symmetries of the system give sufficient conditions for angular frequency synchronization~\cite{Golubitsky2006b}. These are not necessary as there may be other dynamically invariant subspaces where oscillators are phase and frequency locked which are not induced by symmetry but rather by balanced polydiagonals of colored graphs~\cite{Antoneli2006}.

\section{Weak Chimeras for Networks of Phase Oscillators}
\label{sec:WeakChimerasPhaseOsc}
Definition~\ref{def:SymWeakChimera} relates to the original definition of a weak chimera for networks of coupled phase oscillators~\cite{Ashwin2014a}. We will not restrict ourselves to systems~\eqref{eq:COsc} but consider a more general setup that may include, for example, nonpairwise interactions~\cite{Ashwin2015a, Bick2016b}. More precisely, let $\X=\Torn$ and let $\G\subset\Sn$ act transitively on~$\Torn$ by permuting indices. A smooth $\G$-equivariant vector field $Y:\Torn\to\Rn$ now defines a $\G$-equivariant dynamical system
\begin{equation}
\label{eq:PhaseDyn}
\dot\vphi = Y(\vphi).
\end{equation}
that describes the evolution of~$\maxdim$ phase oscillators where the state of oscillator~$k$ is given by $\vphi_k\in\Tor$.

Write $z_k = \exp(i\vphi_k)$ and identify initial conditions~$\vphi^0\in\Torn$ with~$z^0\in\Csn$. The dynamics of~\eqref{eq:PhaseDyn} can be embedded in~$\Csn$ as
\begin{equation}
\label{eq:zPhase}
\dot z_k = z_k(iY_k(\vphi)).
\end{equation}
Therefore
\begin{equation}\label{eq:OmAv}
\Omega_k(\vphi^0) := K_{Y_k}(z^0) = \lim_{T\to\infty}\frac{1}{T}\int_0^T Y_k(\vphi(t))\ud t
\end{equation}
and if $A\subset\Torn$ is compact, dynamically invariant supporting an SRB measure then
\[\Om(\vphi^0) = K_Y(A)\] 
is the \emph{angular frequency vector} for~$A$. Moreover, with~\eqref{eq:PhaseDyn} we have
\begin{equation}
\int_{0}^{T}Y_k(\vphi(t))\ud t
= \int_{0}^{T}\dot\vphi_k(t)\ud t = 
\hat{\vphi}_k(T)-\hat{\vphi}_k(0)
\end{equation}
where $\hat{\vphi}$ is a continuous lift of the trajectory 
$\vphi(t)$ to~$\Rn$.
Thus, 
\[\Omega_k(\vphi^0) = \lim_{T\to\infty}\frac{\hat{\vphi}(T)}{T}\]
as given by~\eqref{eq:Omegak}. Note also that~$\Om_k(A)$ correspond to the average frequency defined in~\cite{Golubitsky2006b} and relates to the rotation vector for torus maps~\cite{Misiurewicz1989}.

Compared to the original definition of a weak chimera in~\cite{Ashwin2014a}, Definition~\ref{def:SymWeakChimera} is more restrictive. More precisely, for~$A$ we require that frequency synchronization is only relevant for a weak chimera if the oscillators are related by symmetry. By contrast, the original definition considers the set
\begin{equation}
\Theta(A) = \set{\gamma\in\Sn}{\gamma\Om(A)=\Om(A)}
\end{equation}
rather than the isotropy $\G_{\Om_k(A)}$. Note that~$\Theta(A)$ may be strictly larger than~$\G_{\Om(A)}$. For example if~$\Zn = \Z/\maxdim\Z\subset\Sn$ denotes the cyclic group and~$X$ is~$\Zn$ equivariant but not $\Sn$-equivariant and $\vphi^0_1=\dotsb=\vphi^0_\maxdim$ (for example a nonlocally coupled ring of phase oscillators~\cite{Kuramoto2002}) is a solution of $\dot\vphi= X(\vphi)$ then $\Theta\big(\sset{\vphi^0}\big) = \Sn\supsetneq\Zn$.


\section{Persistence of Weak Chimeras without Symmetry on Average for Diffusively Coupled Phase Oscillators}
\label{sec:WeakChimNoSym}

The inclusions~\eqref{eq:Inclusion} in Section~\ref{sec:SymImplSync} imply that any (nontrivial) instantaneous or average symmetry of a dynamically invariant set gives nontrivial isotropy of the angular frequency vector. This is the case for the weak chimeras constructed in~\cite{Ashwin2014a, Panaggio2015b, Bick2015c}. In this section we construct weak chimeras with trivial average symmetries for systems consisting of two weakly interacting populations of phase oscillators.

\subsection{Coupling function separability for symmetric diffusively coupled phase oscillators}
For $\phi, \psi\in\Torn$ define $X=(X_1, \dotsc, X_\maxdim)$ by
\begin{equation}
\label{eq:Y}
X_k(\phi, \psi) := \frac{1}{\maxdim}\sum_{j=1}^{\maxdim}g(\phi_k-\psi_j).
\end{equation}
The dynamics of a fully symmetric network of~$\maxdim$ phase oscillators with coupling function~$g$ is given by the $\Sn$-equivariant dynamical system on~$\Torn$ where 
\begin{equation}\label{eq:OscVF}
\dot\vphi_k  = Y_k(\vphi) := X_k(\vphi, \vphi)
\end{equation}
describes the evolution of the~$k$th oscillator\footnote{We obtain~\eqref{eq:OscVF} by setting $H_{kj}=1$ for all $k,j$ in~\eqref{eq:COsc}.}. We may assume~$g(0)=0$ by going to suitable co-rotating reference frame, $\vphi_k \mapsto \vphi_k - {\omega}t$. If the choice of coupling function~$g$ is important we write $Y^{(g)}$ or $X^{(g)}$ to highlight the dependency. Reducing the continuous~$\Tor$ symmetry of~\eqref{eq:OscVF} allows to set $\vphi_1=0$. Because of the~$\Sn$-equivariance, the \emph{canonical invariant region} 
\begin{equation}
\CIR:=\set{\vphi\in\Torn}{0=\vphi_1<\dotsb<\vphi_\maxdim < 2\pi}
\end{equation}
is dynamically invariant. It is bounded by hypersurfaces corresponding to cluster states with $\vphi_k=\vphi_{k+1}$ and there is a residual~$\Zn$ symmetry on~$\CIR$~\cite{Ashwin1992, Ashwin2016}.

For a compact flow invariant set~$A\subset\Torn$ define
\begin{equation}
\Xi(A) = \bigcup_{k\neq j}\set{\vphi_k-\vphi_j}{\vphi\in A}.
\end{equation}
Note that $\Xi(\g A)=\Xi(A)$ for all $\g\in\Sn$.

\begin{defn}\label{def:CouplFuncSep}
Two sets $A_1, A_2\subset\CIR$ are \emph{coupling function separated} if there are open intervals $Q_{A_1}, Q_{A_2}\subset \Tor$ with $\Xi(A_\ell)\subset Q_{A_\ell}$, $\ell=1,2$, and 
\[\overline{Q}_{A_1}\cap\overline{Q}_{A_2}=\emptyset\]
where the bar denotes topological closure.
\end{defn}

For $Q\subset\Tor$ define $\Xi^{-1}(Q):=\set{\vphi\in\Tor}{\Xi(\sset{\vphi})\subset Q}$ and
\[W^{(g)}(Q) := \Big[\min_{k\in\sset{1, \dotsc, \maxdim}}\inf_{\vphi\in \Xi^{-1}(Q)} Y_k^{(g)}(\vphi), \max_{k\in\sset{1, \dotsc, \maxdim}}\sup_{\vphi\in \Xi^{-1}(Q)} Y_k^{(g)}(\vphi)\Big].\]

\begin{lem}\label{lem:CplngFuncSep}
\begin{enumerate}
\item\label{lem:CFSone}
If $A\subset\CIR$ with $\Xi(A)\subset Q$ is dynamically invariant for the dynamics of~\eqref{eq:OscVF} with coupling function~$g$ then $\Om_k(\vphi^0)\in W^{(g)}(Q)$.
\item\label{lem:CFStwo}
Suppose that $A_\ell\subset\CIR$, $\ell=1,2$, are compact and coupling function separated with separating sets $Q_{A_\ell}$. Then for any $\eta\geq 0$ we can find a coupling function~$\gh$ such that
\[B_\eta\left(W^{(\gh)}(Q_{A_1})\right)\cap B_\eta\left(W^{(\gh)}(Q_{A_2})\right)=\emptyset.\]
\item\label{lem:CFSthree}
Let $A_\ell$ be as above and let $A_1',A_2'\subset\CIR$ be dynamically invariant for the dynamics of~\eqref{eq:OscVF} with $\Xi(A_\ell')\subset Q_{A_\ell}$. Then there is a coupling function~$\gh$ such that $A_1',A_2'$ are dynamically invariant for the dynamics of~\eqref{eq:OscVF} with~$\gh$ and
\[\Om_k\big(\vphi^0_{A_1'}\big)\neq\Om_j\big(\vphi^0_{A_2'}\big)\]
for all $k,j$ and $\vphi^0_{A_\ell'}\in A_\ell'$.
\end{enumerate}
\end{lem}

\begin{proof}
To prove (\ref{lem:CFSone}) note first that invariance of~$\CIR$ implies that $\Om_j(\vphi^0)=\Om_k(\vphi^0)$ for all $k,j$ and $\vphi^0\in A$. Standard integral estimate for~\eqref{eq:OmAv} yield
\[\Om_k(\vphi^0)\in \Big[\inf_{\vphi\in A} Y_k^{(g)}(\vphi), \sup_{\vphi\in A} Y_k^{(g)}(\vphi)\Big] \subset W^{(g)}(Q).\]

To prove (\ref{lem:CFStwo}) consider coupling functions~$\gh$ with $\gh(\phi)=g(\phi)+a_\ell$ for all $\phi\in Q_{A_\ell}$, $\ell=1,2$. Since \[Y_k^{(\gh)}(\vphi) = \frac{1}{\maxdim}\sum_{j=1}^{\maxdim}(g(\phi_k-\psi_j)+a_\ell) = a_\ell + Y_k^{(g)}(\vphi)\] for all $\vphi$ with $\Xi(\sset{\vphi})\subset Q_{A_\ell}$ we have 
\[W^{(\gh)}(Q_{A_\ell}) = \left[\inf W^{(g)}(Q_{A_\ell}) + a_\ell, \sup W^{(g)}(Q_{A_\ell}) + a_\ell\right].\]
For a given $\eta\geq0$ choose $a_1, a_2$ such that
\[B_\eta\left(W^{(g)}(Q_{A_1})\right)\cap B_\eta\left(W^{(g)}(Q_{A_2})\right)=\emptyset\]
to obtain the desired coupling function~$\gh$.

Note that replacing $g$ by~$\gh$ as above preserves dynamically invariant sets~$A$ with $\Xi(A)\subset Q_{A_\ell}$. Claim~(\ref{lem:CFSthree}) now follows from (\ref{lem:CFSone}) and (\ref{lem:CFStwo}) with $\eta=0$.\end{proof}

\begin{rem}
The notion of function coupling separability and Lemma~\ref{lem:CplngFuncSep} generalize to a finite number of sets $A_1, \dotsc, A_r$. The function $g-\gh$ can typically be chosen to be~$\Cinf$.
\end{rem}

\subsection{Relative equilibria with trivial symmetry}
We now show that choosing the coupling function appropriately in an arbitrarily small neighborhood of zero gives rise to asymptotically stable relative equilibria with trivial symmetry for~\eqref{eq:OscVF}.

Let $0=\alpha_1 <\dotsb <\alpha_\maxdim<2\pi$. The function
\begin{equation}
\label{eq:RelEq}
\vphi^\star(t)=(\alpha_1+t{\omega^\star}, \dotsc, \alpha_\maxdim+t{\omega^\star})\in\CIR
\end{equation}
with $\omega^\star = \frac{1}{n}\sum_{j=2}^\maxdim g(-\alpha_j)$ is a relative equilibrium of~\eqref{eq:OscVF} for any coupling function~$g$ such that
\begin{equation}
\label{eq:RelEqExistence}
\frac{1}{\maxdim}\sum_{j\neq k}\left(g(\alpha_k-\alpha_j)-g(-\alpha_j)\right) = 0
\end{equation}
for all $k=2, \dotsc, \maxdim$. For a relative
equilibrium we have 
\begin{equation}
\Xi(\sset{\vphi^\star})=\bigcup_{k\neq j}\sset{\alpha_k-\alpha_j}\subset [-2\alpha_{\maxdim}, 2\alpha_{\maxdim}]\subset \Tor.
\end{equation}
In particular, we have a relative equilibrium if the coupling function~$g$ vanishes on $\Xi(\sset{\vphi^\star})$. Since~$\alpha_n$ can be chosen arbitrarily small, the relative equilibrium can be chosen arbitrarily close to the fully synchronized solution $\vphi_1=\dotsb=\vphi_\maxdim$. We have $\Om(\sset{\vphi^\star}) = (\omega^\star, \dotsc, \omega^\star)$.

Stability of the relative equilibrium is determined by the linearization
\begin{equation}
\frac{\partial Y_k}{\partial\vphi_j} = \begin{cases}
\frac{1}{n}\sum_{l\neq k}g'(\alpha_k-\alpha_l) & \text{if } k=j,\\
-\frac{1}{n}g'(\alpha_k-\alpha_j) & \text{otherwise.}
\end{cases}
\end{equation}
By choosing the coupling function appropriately on~$\Xi(A)$ the relative equilibrium will be asymptotically stable. For example, if $g'(\phi)=0$ for $\phi<0$ and $g'(\phi)<0$ for $\phi>0$ we have a lower triangular matrix with negative values on the diagonal (apart from one zero eigenvalue) implying that~$\vphi^\star$ is asymptotically stable.

\begin{lem}\label{lem:EqInC}
Let $\vphi^\star(t)$ be a relative equilibrium of~\eqref{eq:OscVF} 
as defined in~\eqref{eq:RelEq}.
If~$\abs{\alpha_\maxdim}<\frac{\pi}{2}$ then
$T(\sset{\vphi^\star})=\Sg(\sset{\vphi^\star})=\TG$.
\end{lem}

\begin{proof}
It suffices to show that $\Sg(\sset{\vphi^\star})=\TG$.
Assume that $\g\in\Sg(\sset{\vphi^\star})$ 
with $\g\neq\id$.
Then there exists a $\tau\geq 0$ such that
\[\alpha_{\g k} = \alpha_{k}+\tau\omega^\star \mod 2\pi\]
for all~$k$. Recall that $0=\alpha_1 <\dotsb <\alpha_\maxdim<2\pi$. Since $\gamma\neq\id$, $\gamma$ permutes some indices. Assume without loss of generality that $\alpha_{\g 1}>\alpha_1$ and $\alpha_{\g 2}<\alpha_2$. We have $\alpha_{\g 1}-\alpha_1 = \alpha_{\g 2}-\alpha_2 = \tau\omega^\star\mod 2\pi$. But since $\alpha_{\g 1}-\alpha_1>0$ and $\alpha_{\g 2}-\alpha_2<0$ there has to be an $m>0$ such that $\alpha_{\g 1}-\alpha_1 -\alpha_{\g 2}+\alpha_2 = 2m\pi$. This is a contradiction since $\abs{\alpha_{\g 1}-\alpha_1 - \alpha_{\g 2}+\alpha_2}\leq 4\alpha_\maxdim<2\pi$.
\end{proof}

\subsection{Weak chimeras with $\Sg(A)=\TG$ in weakly coupled populations of phase oscillators}
Chaotic weak chimeras have many features associated with classical chimera states including positive maximal Lyapunov exponents. Hence, rather than using a hyperbolicity argument to construct nonchaotic weak chimeras as in~\cite{Ashwin2014a}, we aim to construct weak chimeras with $\Sg(A)=\TG$ in a more general setup which allows for positive maximal Lyapunov exponents. To this end, we extend recent results from~\cite{Bick2015c} with respect to the instantaneous and setwise symmetries of the constructed sets.

Coupling two populations of~$\maxdim$ oscillators, whose uncoupled dynamics are given by~\eqref{eq:OscVF}, defines a dynamical system on $\Tornn$. More explicitly,
write $\vphi = (\vphi_{1}, \vphi_{2})\in\Torn\times\Torn=\Tornn$, $\vphi_{\ell}=(\vphi_{\ell,1}, \dotsc, \vphi_{\ell,\maxdim})$ and consider the product system
\begin{equation}\label{eq:ProductDyn}
\begin{aligned}
\dot\vphi_{1} &= \F_1^{(g,\e)}(\vphi_{1}, \vphi_2) = Y^{(g)}(\vphi_{1}) + \e X^{(g)}(\vphi_{1},\vphi_{2}),\\
\dot\vphi_{2} &= \F_2^{(g,\e)}(\vphi_{1}, \vphi_2) = Y^{(g)}(\vphi_{2}) + \e X^{(g)}(\vphi_{2},\vphi_{1}),
\end{aligned}
\end{equation}
with $Y^{(g)}, X^{(g)}$ as in~\eqref{eq:OscVF},~\eqref{eq:Y}. Observe that for $\e=0$ the system decouples into two identical groups of~$n$ oscillators---both of which with nontrivial dynamics~\eqref{eq:OscVF}. For $\vphi^0\in\Tornn$ we denote the asymptotic angular frequency of the oscillator with phase $\vphi_{\ell,k}$ by $\Om_{\ell,k}(\vphi^0)=\Om_{\ell,k}^{(g,\e)}(\vphi^0)$.

Let $\Gamma = \Sn\wr\Sk{2}$ where~$\wr$ is the wreath product. The system~\eqref{eq:ProductDyn} is $\G$-equivariant~\cite{Dionne1996a}; we have $\G = \Sn\wr\Sk{2}=(\Sn)^2\rtimes\Sk{2}$ where the elements of~$\Sn$ permute the oscillators within each group of~$n$ oscillators and the action of~$\Sk{2}$ permutes the two groups. Observe that this is only a semidirect product~$\rtimes$ as the two sets of permutations do not necessarily commute. The oscillators are indistinguishable as this group acts transitively on the oscillators. 

Weak chimeras in the product system persist for weak coupling~$0\leq\e\ll 1$. As in~\cite{Bick2015c}, we call a dynamically invariant set~$A$ is \emph{sufficiently stable} if is there is an open neighborhood of~$A$ on which a Lyapunov function is defined. The persistence theorem for weak chimeras~\cite[Theorem~4]{Bick2015c} generalizes to coupling function separated sets that are sufficiently stable.

\begin{thm}
\label{thm:CWCns}
Suppose that~$g$ is a coupling function such that $A_1, A_2\subset\CIR$ are compact, forward invariant, coupling function separated and sufficiently stable sets for the dynamics of~\eqref{eq:OscVF} with~$Y^{(g)}$. Then for any sufficiently small $\delta>0$ there exist a smooth coupling function~$\gh$ and $\e_0>0$ such that for any $0\leq\e<\e_0$ the weakly coupled product system~\eqref{eq:ProductDyn} with~$g$ replaced by~$\gh$ has a sufficiently stable weak chimera~$A^{(\e)}$ with $A^{(\e)} \subset B_\delta(A_1\times A_2)$.
\end{thm}

\begin{proof}
First consider the coupling function separated sets $A_1, A_2\subset\Torn$ as dynamically invariant sets for~\eqref{eq:OscVF}, a factor of the uncoupled system. Suppose that $Q_{A_1}, Q_{A_2}$ are the separating sets and for a coupling function $\gt$ define $M(\gt) := \max_{(\vphi_1, \vphi_2)\in\Tornn}\tabs{X^{(\gt)}(\vphi_1, \vphi_2)}<\infty$. Now choose a coupling function~$\gh$ according to Lemma~\ref{lem:CplngFuncSep}(\ref{lem:CFStwo}) for $\eta=1$ and fix~$\e_1:=M(\gh)^{-1}$. For any $0\leq\e<\e_1$ we have
\begin{equation}\label{eq:Sep}B_{\e M(\gh)}(W^{(\gh)}(Q_{A_1}))\cap B_{\e M(\gh)}(W^{(\gh)}(Q_{A_2}))=\emptyset.\end{equation}
since $\e M(\gh)<\e_1 M(\gh)=1$.

Now consider the product system~\eqref{eq:ProductDyn}. For any sufficiently small $\delta>0$ we obtain $\e_2>0$ and compact invariant sets~$A^{(\e)}\subset B_\delta(A_1\times A_2)$ for all $0\leq\e<\e_2$ as in~\cite{Bick2015c}. Set $\e_0 < \min\sset{\e_1, \e_2}$. By restricting $\delta$ appropriately, the sets~$A^{(\e)}$ are weak chimeras.

To show that $\G_{\Om(\vphi^0)}\neq\TG$, assume that~$\delta$ is so small that $B_\delta(A_1\times A_2)\subset\CIR^2$. This implies that the phase ordering within each population is preserved. Hence, for given $\ell=1,2$ we have
\[
\Om_{\ell,k}^{(\gh,\e)}(\vphi^0) = \Om_{\ell,j}^{(\gh,\e)}(\vphi^0) =: \Om_{\ell,*}^{(\gh,\e)}(\vphi^0)
\]
for all $k,j$ and any $\vphi^0\in A^{\e}$. Thus $\G_{\Om(\vphi^0)}\neq\TG$.

It remains to be shown that $\G_{\Om(\vphi^0)}\neq\G$. Let $A_\ell^{(\e)}$ denote the projection of~$A^{(\e)}$ onto $\vphi_\ell$. Now assume that~$\delta$ is sufficiently small such that $A^{(\e)} \subset B_\delta(A_1\times A_2)$ implies that $\Xi\big(A_\ell^{(\e)}\big)\subset Q_{A_\ell}$ for all $0\leq\e<\e_0$. Since $\F_1^{(\gh,\e)}(\vphi_{1}, \vphi_2) = Y^{(\gh)}(\vphi_{1}) + \e X(\vphi_{1},\vphi_{2})$ integral estimates as in Lemma~\ref{lem:CplngFuncSep}(\ref{lem:CFSone}) imply that
\[\Om_{\ell,*}^{(\gh,\e)}\in B_{\e_0M}\big(W^{(\gh)}(Q_{A_\ell})\big)\]
for all $0\leq \e<\e_0$. By choice of~$\gh$, Equation~\eqref{eq:Sep} now implies $\Om_{1,*}^{(\gh,\e)}(\vphi^0)\neq\Om_{2,*}^{(\gh,\e)}(\vphi^0)$ for all $\vphi^0\in A^{(\e)}$. Thus $\G_{\Om(\vphi^0)}\neq\G$ and~$A^{(\e)}$ is a weak chimera.
\end{proof}

The following statement asserts that trivial symmetries in the
factors carry over to the product dynamics.

\begin{lem}
\label{lem:SymProd}
Let $A_1, A_2\subset \Torn$ be coupling function separated attractors 
for~\eqref{eq:OscVF} with $\Sg(A_1)=\Sg(A_2)=\TG$. Then $\Sg(A_1\times A_2)=\TG$ 
for the product system~\eqref{eq:ProductDyn}.
\end{lem}

\begin{proof}
Write $\Sk{2} = \sset{\id, \tau}$. For any $\gamma\in(\Sn)^2$ we have $(\g, \id)(A_1\times A_2)\cap(A_1\times A_2) = \emptyset$ in~$\Tornn$ by assumption. Since~$A_1$ and~$A_2$ are coupling function separated, we have $A_1\cap A_2=\emptyset$ in~$\Torn$. Write $\Gamma V = \bigcup_{\g\in\G}\gamma V$, $V\in\sset{A_1,A_2}$. The fact that $\Xi(\g A_1)=\Xi(A_1)$ implies $\G A_1\cap \G A_2=\emptyset$ in~$\Torn$. Since $(\g, \tau)(A_1\times A_2)\subset\G A_2\times\G A_1$ for any $\gamma\in(\Sn)^2$ we have $(\g, \tau)(A_1\times A_2)\cap(A_1\times A_2)=\emptyset$ and by Proposition~\ref{prop:Symmetry} the claim follows.
\end{proof}

Combining the perturbation result Theorem~\ref{thm:CWCns} with the symmetry considerations, we can now state the main theorem of this section. In order to apply Theorem~\ref{thm:CWCns} we make a slightly stronger assumption concerning stability of the relative periodic orbits.

\begin{thm}
\label{thm:ChaoticWeakChimNoSym}
Suppose that~$g$ is a coupling function such that for the~$\Sn$-equivariant dynamics of~\eqref{eq:OscVF} with~$Y^{(g)}$ the set $\Ac=\sset{\vphi^{\star}}$ is a sufficiently stable relative equilibrium and~$\Ai\subset\Torn$ is sufficiently stable attractor that are coupling function separated and $\Sg(\Ac)=\Sg(\Ai)=\TG$. Then for any sufficiently small $\delta>0$ there is a coupling function~$\gh$ and $\e_0>0$ such that for any $0\leq \e<\e_0$ there is a weak chimera~$A^{(\e)}\subset B_\delta(\Ac\times\Ai)$ with $\Sg(A^{(\e)})=\TG$ for the $\Sn\wr\Sk{2}$-equivariant dynamics of~\eqref{eq:ProductDyn} with~$\gh$.
\end{thm}

\begin{proof}
Suppose that $\Sg(\Ac)=\Sg(\Ai)=\TG\subset\Sn$. By Lemma~\ref{lem:SymProd} we have $\Sg(\Ac\times\Ai)=\TG\subset\Sn\wr\Sk{2}$. Since $\Ac\times\Ai$ is assumed to be an attractor, Corollary~\ref{cor:NoSymPerturb} implies that there exists a~$\delta_0>0$ such that for any invariant set $A\subset B_{\delta_0}(\Ac\times\Ai)$ we have $\Sg(A)=\TG\subset\Sn\wr\Sk{2}$.

For sufficiently small $0<\delta<\delta_0$ Theorem~\ref{thm:CWCns} yields an~$\e_0>0$ and weak chimeras $A^{(\e)}\subset B_\delta(\Ac\times\Ai)$ for all $0\leq\e<\e_0$. Since $\delta<\delta_0$ the argument above implies that $\Sg(A^{(\e)})=\TG$ for all such~$\e$.
\end{proof}

\begin{rem}
\begin{enumerate}
\item In fact, the condition that $\Ai$ is a relative equilibrium is not necessary. Theorem~\ref{thm:ChaoticWeakChimNoSym} holds for any compact, sufficiently stable attractor~$\Ac\subset\CIR$ that is coupling function separated from~$\Ai$. Moreover, the same statement holds for (sufficiently unstable) repellers.
\item Even if the weak chimera $A^{(0)}=\Ac\times\Ai$ is observable, extra assumptions on the persistence of SRB measures are needed to prove that $A^{(\e)}$ is an observable weak chimera.
\end{enumerate}
\end{rem}


\newcommand{\inttime}{2\cdot10^5}

\section{A Numerical Example of a Chaotic Weak Chimera with Trivial Symmetry}
\label{sec:Example}
We now give an explicit example of a coupling function such that the dynamics of the product system~\eqref{eq:ProductDyn} give rise to a chaotic weak chimera with trivial symmetry for $\e>0$ following the construction described in the previous section. In contrast to the examples in~\cite{Bick2015c}, the main focus here is on the symmetries of the weak chimeras which we calculate explicitly.

Recall that the dynamics of~\eqref{eq:OscVF} for $\maxdim=4$ oscillators give rise to chaotic attractors~$A$ with~$\Sg(A)=\TG$~\cite{Bick2011, Bick2012b}. Define
\begin{equation}
\label{eq:gchaos}
{g}(\phi) = \sum_{r=0}^{4} c_r \cos (r\phi+\xi_r)
\end{equation}
with $c_1 = -2$, $c_2 = -2$, $c_3 = -1$, and $c_4 = -0.88$. For $\xi_1 = \eta_1$, $\xi_2 = -\eta_1$, $\xi_3 = \eta_1+\eta_2$, and $\xi_4 = \eta_1+\eta_2$ with $\eta_1=0.138$, $\eta_2=0.057511$ the dynamics of~\eqref{eq:OscVF} with this particular choice of coupling function~${g}$ give rise to a chaotic attracting set $\Ai\subset\CIR$ with positive maximal Lyapunov exponents and $\Sg(\Ai) = \TG$. For~$\Ai$ we have $\Xi(\Ai)\subset[0.4, 2\pi-0.4]$ as shown in Figure~\ref{fig:XiAttr}.

\begin{figure}
\includegraphics[scale=\imagescaling]{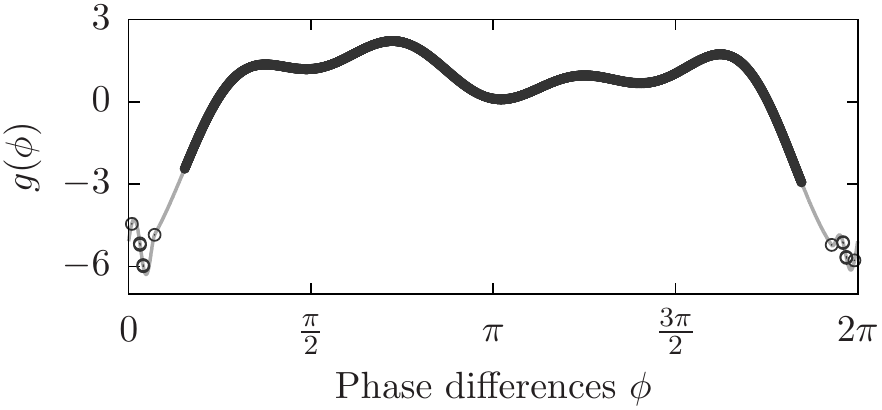}
\caption{\label{fig:XiAttr}
The attracting sets~$\Ai$ and~$\Ac$ of the dynamics given 
by~\eqref{eq:OscVF} with coupling function~$\gh$ are coupling 
function separated.
The coupling function~$\gh$ as defined in~\eqref{eq:gh} is depicted 
by a gray line. The values of~$g$ on~$\Xi(\Ai)$ are indicated by 
filled circles and on $\Xi(\Ac)$ by $12$ hollow circles.}
\end{figure}

A suitable local perturbation of the coupling function~$g$ yields bistability between~$\Ai$ and a relative equilibrium with trivial symmetry in the system defined by~\eqref{eq:OscVF}. Let
\begin{equation}
\label{eq:TripPerturb}
\tilde{g}(\phi) = \sum_{r=6}^{24} a_r \cos (r\phi+\zeta_r)
\end{equation}
with parameters $a_r$, $\zeta_r$ as given in Appendix~\ref{app:CouplingFunc}. Moreover, define
\[
\Bm(x) := \begin{cases}\exp\fleft(-\frac{1}{1-x^2}\right)&\text{if} -1<x<1,\\
0& \text{otherwise}\end{cases}
\]
and let $a\in\R$, $b \in (0, \pi)$ be parameters. Now define $\Bm_{ab}(\phi) := a \Bm\big(\frac{\phi}{b}\big)$ with~$\phi$ taken modulo~$2\pi$ with values in $(-\pi, \pi]$ is a $2\pi$-periodic ``bump function.''
Fix $a = 2.5$, $b=0.25$. Define the $\Cinf$ function
\begin{equation}
\label{eq:gh}
\gh := g + \tilde{g}\Bm_{ab}.
\end{equation}
We have $\gh(\phi) = g(\phi)$ for all $\phi\in [b, 2\pi-b]$. Since $\Xi(\Ai)\subset[b, 2\pi-b]$ we have $Y^{(\gh)}(\vphi)=Y^{(g)}(\vphi)$ for all $\vphi\in\Ai$. Thus~$\Ai\subset\CIR$ is also a chaotic attracting set for the dynamics of~\eqref{eq:OscVF} with coupling function~$\gh$. In addition, there is a stable relative periodic orbit $\vphi^\star(t) \approx (t{\omega^\star}, 0.0975+ t{\omega^\star}, 0.1253+ t{\omega^\star}, 0.2247+ t{\omega^\star})$. For $\Ac=\set{\vphi^\star(t)}{t\geq 0}$ we have $\Xi(\Ac)\subset[-0.3, 0.3]$. Therefore, the sets $\Ai$ and~$\Ac$ are coupling function separated; see Figure~\ref{fig:XiAttr}.

\begin{figure}
\subfloat[Phase evolution]{\includegraphics[scale=\imagescaling]{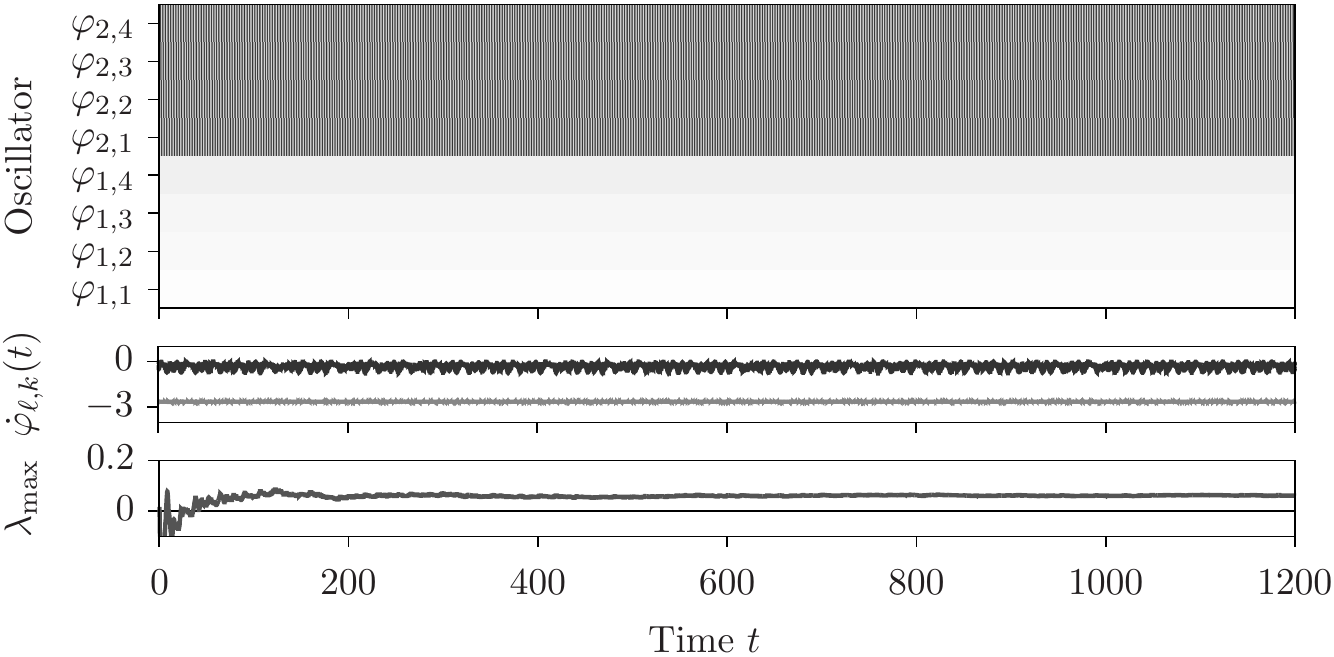}}\\
\subfloat[Projected dynamics of each population]{\includegraphics[scale=\imagescaling]{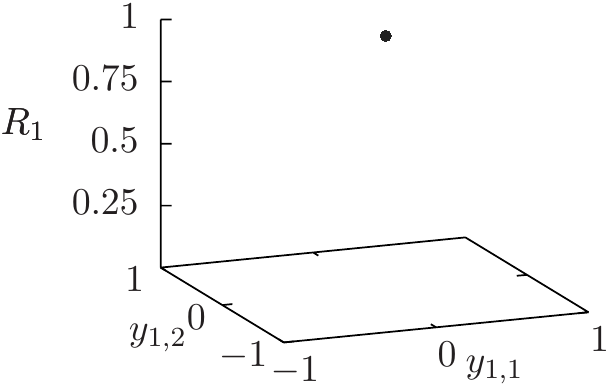}\quad\quad\includegraphics[scale=\imagescaling]{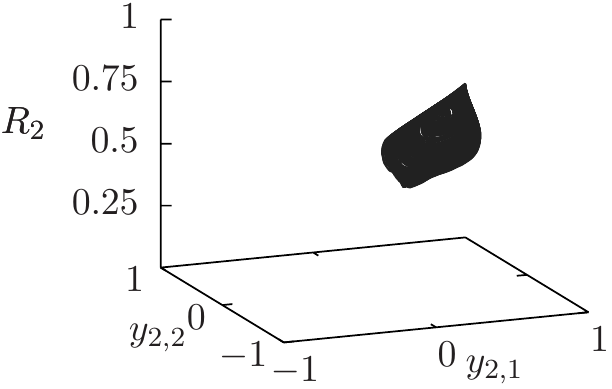}}
\caption{\label{fig:PerturbAttr}
Chaotic weak chimeras with trivial setwise symmetries appear in
the $\Sk{4}\wr\Sk{2}$-equivariant system~\eqref{eq:ProductDyn} with two populations of $\maxdim=4$ oscillators for $\e=0.01$.
Panel~(\textsc{a}) shows the phase evolution: the phase of the oscillators
(periodic color scale, $\vphi_{\ell,k}(t)=0$ in black and $\vphi_{\ell,k}(t)=\pi$ in white) in a co-rotating frame at the speed of the first oscillator is shown at the top, the instantaneous frequencies $\dot\vphi_{\ell,k}(t)$ in the middle, and convergence of the maximal Lyapunov exponent at the bottom. Panel~(\textsc{b}) shows the dynamics on the attracting set for each population in the $\Z_4$-equivariant projections $y_\ell = (\sin(\vphi_{\ell,3}-\vphi_{\ell,1}), \sin(\vphi_{\ell,4}-\vphi_{\ell,2}))$ where~$\Z_4$ are the permutations within populations that preserve the phase ordering.
}
\end{figure}

Now consider two weakly coupled populations~\eqref{eq:ProductDyn} of $\maxdim=4$ oscillators. Since $\Sg(\Ac)=\Sg(\Ai)=\TG$ we have that $\Sg(\Ac\times\Ai)=\TG$ for $\e=0$ and we expect dynamically invariant sets with trivial symmetry for small $\e>0$. We integrated system~\eqref{eq:ProductDyn} numerically\footnote{Integration was carried out in MATLAB using the variable order scheme \texttt{ode113} with a adaptive time step $\Delta t\leq 10^{-1}$ subject to conservative relative and absolute error tolerances of~$10^{-9}$ and~$10^{-11}$ respectively.} and calculated the maximal Lyapunov exponent from the variational equations. The attracting set~$A^{(\e)}$ for $\e=0.01$ close to $\Ac\times\Ai$ with trivial setwise symmetries and positive maximal Lyapunov exponent is shown in Figure~\ref{fig:PerturbAttr}; the absolute value of the local order parameter $R_\ell(t)=\tabs{\frac{1}{4}\sum_{j=1}^4\exp(i\vphi_{\ell,j})}$ gives information about the synchronization of each population: it is equal to one if all oscillators within the populations are phase synchronized.

\begin{figure}
\subfloat[Attractor in the $\Sk{4}$-equivariant projection~$y_\ell$]{\includegraphics[scale=\imagescaling]{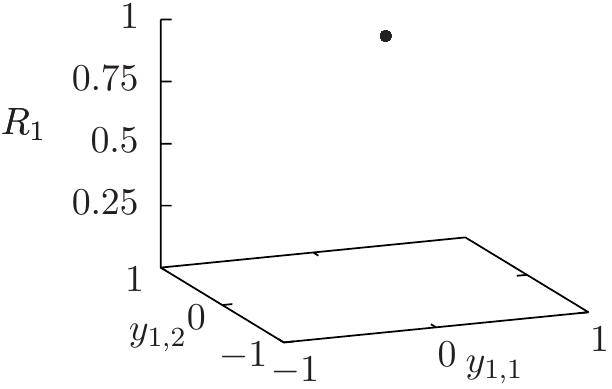}\quad\quad\includegraphics[scale=\imagescaling]{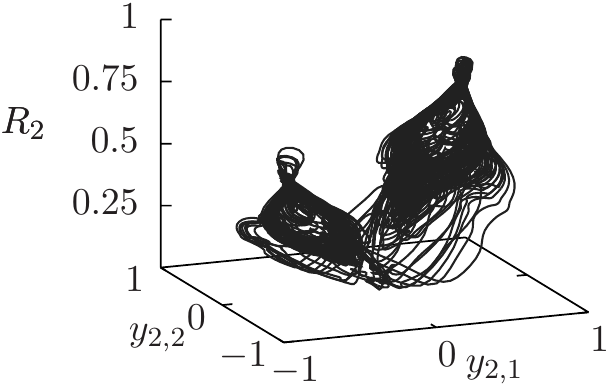}}\\
\subfloat[Maximal Lyapunov exponents and symmetries with varying $\e$]{\includegraphics[scale=\imagescaling]{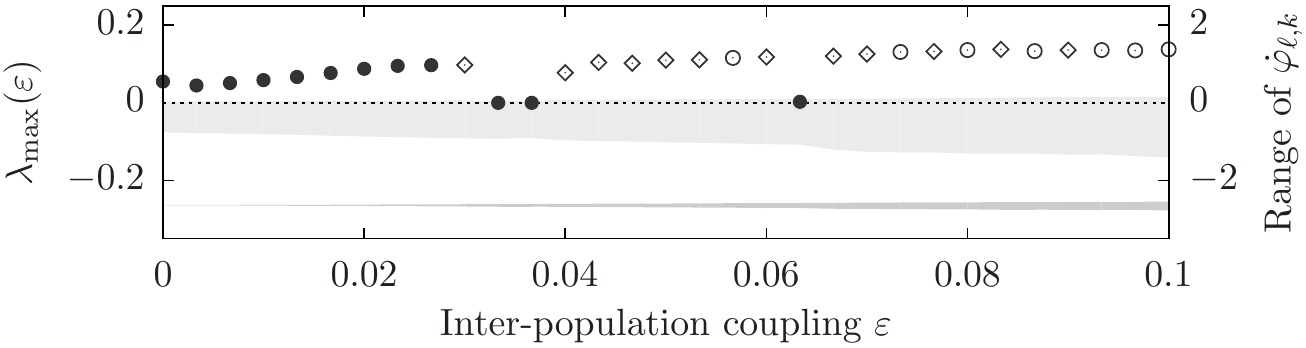}}
\caption{\label{fig:LyapScan}
Increasing~$\e$ yields chaotic weak chimeras that undergo symmetry increasing bifurcations. 
Panel~(\textsc{a}) shows a trajectory for~$\e=0.1$ converging to an attractor $A^{(\e)}$ with $\Sg(A^{(\e)})\neq\TG$.
Panel~(\textsc{b}) shows maximal Lyapnuov exponents obtained by integrating~\eqref{eq:ProductDyn} from a fixed initial condition on $A^{(0)}$ for $T=\inttime$ time units. The marker indicates the symmetry of the attractor~$A^{(\e)}\subset\Tornn$: ``$\bullet$'' for $\Sg(A^{(\e)})=\TG$, 
``$\diamond$'' if $\Sg(A^{(\e)})\stackrel{?}{=}\TG$, and ``$\circ$'' if $\Sg(A^{(\e)})\neq\TG$. The shaded regions show the intervals $[\min_{k,t}\dot\vphi_{\ell,k}(t), \max_{k,t}\dot\vphi_{\ell,k}(t)]$ for $\ell=1$ (dark gray) and $\ell=2$ (light gray)---where these do not overlap, there is no frequency synchronization between the two populations and hence a weak chimera. }
\end{figure}

For increasing coupling parameter~$\e$ (while keeping the initial condition fixed) the symmetries of the attracting chaotic weak chimeras~$A^{(\e)}$ change; cf.~Figure~\ref{fig:LyapScan}. We integrated the system for $T=\inttime$ time units to calculate both the maximal Lyapunov exponents and detect the presence of nontrivial symmetries. For $A^{(\e)}\subset\CIR^2$ we have to check for permutations of oscillators that preserve the ordering of the phases within each population to determine the symmetry of the attractor. To this end, we calculated the ergodic average $S_\ell(\e) = \int_0^T\sin(\vphi_{\ell,3}(t)-\vphi_{\ell,1}(t))\ud t$ along the trajectory which converges zero if $\Sg(A^{(\e)})\neq\TG$. Note that if symmetric copies of attractors merge in a symmetry increasing bifurcation~\cite{Chossat1988}, these ergodic averages may converge very slowly. Previous numerical investigations of the chaotic attractor in the uncoupled system~\cite{Bick2012b} showed that attractors with trivial symmetry are confined to one quadrant under the projection $y_\ell = (\sin(\vphi_{\ell,3}-\vphi_{\ell,1}), \sin(\vphi_{\ell,4}-\vphi_{\ell,2}))$. Thus, the number of quadrants $Q_\ell(\e)$ that the projected trajectory enters being greater than one indicates that a symmetry increasing bifurcation may have occurred; compare also Figures~\ref{fig:PerturbAttr}(b) and~\ref{fig:LyapScan}(a). Consequently, we conclude $\Sg(A^{\e})=\TG$ if~$\abs{S_2(\e)}>10^{-1}$---but we write $\Sg(A^{\e})\stackrel{?}{=}\TG$ if $Q_2(\e)>1$ at the same time to indicate that a symmetry increasing bifurcation may have happened already---and $\Sg(A^{\e})\neq\TG$ otherwise.

Further numerical investigation shows that there is multistability for $\e\geq 0$; the attracting sets~$A^{(\e)}$ for $\e>0$ may coexist with other attracting solutions (not shown).

\section{Discussion}
\label{sec:Discussion}
If a dynamical system has permutational symmetry~$\G$, what are the symmetry properties of the asymptotic angular frequencies which describe how trajectories wind around phase space? For the dynamical systems on~$\Csn$ considered in Section~\ref{sec:WeakChimeras}, the asymptotic angular frequencies are given by averages of $\G$-equivariant observables. This observation yields a natural reformulation of the notion of a weak chimera in terms of the isotropy of the vector of asymptotic angular frequencies. Our definition is not only compatible with the action of~$\G$ but also goes beyond phase oscillators in the weak coupling limit: it applies to more general oscillator models where chimera states have been reported~\cite{Sethia2013, Zakharova2014}. With a rigorous definition in place, it would be desirable to prove the existence of weak chimeras in such systems and show that the dynamics observed are persistent phenomena. These ideas equally apply to more general spaces~$\X$ with a symmetry group acting on it; here asymptotic winding numbers of asymptotic cycles describe the rotation of a trajectory with respect to topological properties of~$\X$~\cite{Schwartzman1957, Fried1982, Walsh1995}. Precisifying the notion of a weak chimera for general topological spaces~$\X$ with symmetry is beyond the scope of the current paper and will be addressed in future work.

Using coupling functions that give rise to relative equilibria with trivial symmetry, we showed that for symmetric phase oscillator systems that there are indeed weak chimeras that have symmetries in the frequencies that are not present in the solutions. This motivates some further symmetry related questions. For example, what are the possible isotropy groups of the angular frequency vector for a $\Gamma$-equivariant system that do not arise from the symmetries of the solutions themselves? (These are obviously restricted to subgroups of the symmetry group.) Which symmetry increasing bifurcations happen as the inter-population coupling~$\e$ is increased (Figure~\ref{fig:LyapScan})? While chaotic dynamics do persist up to $\e\approx 0.1$ (and for other choices of coupling function even up to $\e\approx 0.3$~\cite{Bick2015c}), chaotic weak chimeras with trivial symmetry only persist for values of~$\e$ close to zero. Thus, are there chaotic weak chimeras with trivial symmetry for ``strongly coupled'' populations of phase oscillators? Moreover, 
in general there will be more than one ergodic invariant measure supported on a weak chimera. For each of these measures we obtain asymptotic angular frequency vectors that potentially have different isotropy. While the set of all measures supported on the invariant set of interest~\cite{Jenkinson2006} yields bounds of the asymptotic angular frequencies (see also~\cite{Bick2015c}), a more detailed understanding what the specific isotropy subgroups for the invariant measures are and how they bifurcate would be desirable.

It is also worth noting that asymptotic angular frequencies as averages and their isotropy may still be well be defined if the permutational symmetry of the system is broken due to a (small) perturbation. However, care has to be taken to extend the notion of a weak chimera to nearly symmetric systems since symmetry breaking can have drastic effects on frequency synchronization~\cite{Ashwin2006}. 

``Classical'' chimera states were first observed on rings of nonlocally coupled phase oscillators~\cite{Kuramoto2002}. A finite-dimensional approximation yields a dynamical system that is equivariant with respect to the action of the dihedral group~\cite{Ashwin1992}. Roughly speaking, classical chimeras on finite dimensional rings are trajectories that show characteristic angular frequency synchronization for some finite time as they exhibit pseudo-random drift along  the ring before converging to the fully synchronized state~\cite{Omelchenko2010, Wolfrum2011b}. These are not weak chimeras in the sense defined above. By contrast, initial conditions in the (dynamically invariant) fixed point spaces of a reflection symmetry yield symmetric solutions that eventually converge to the fully synchronized state~\cite{Omelchenko2015}, resembling a transient weak chimera. Interestingly, the chaotic weak chimeras constructed here share an important feature with these ``classical'' chimera states: the isotropy of the angular frequency vector may be larger than the symmetry of the solution itself as the oscillators in the ``coherent'' region are never perfectly phase synchronized. Thus, clarifying the relationship between classical chimera states on rings and the symmetry of the system further---also with respect to the symmetries of the system that describes the continuum limit---provides exciting directions for future research.

\section*{Acknowledgements}
The author would like to thank the anonymous referees for their feedback which significantly helped to improve the presentation of the results. Moreover, the author would like to thank Peter Ashwin, Erik Martens, and Oleh Omel'chenko for stimulating discussions and critical feedback on the manuscript. The research leading to these results has received funding from the People Programme (Marie Curie Actions) of the European Union's Seventh Framework Programme (FP7/2007--2013) under REA grant agreement no.~626111.

\appendix

\section{A trigonometric polynomial coupling function}
\label{app:CouplingFunc}
The Fourier coefficients of~\eqref{eq:TripPerturb} in 
Section~\ref{sec:Example} are given by
\begin{align*}
a_6&=-0.676135392447403 &   \zeta_6&=0.846647746060342\\
a_8&=0.844660333390606 &  \zeta_8&=0.954985847962987   \\
a_{10}&=0.087624615584542 &\zeta_{10} &= 0.212748482509925  \\
a_{12}&=-0.644961491438887 &\zeta_{12} &= 0.025296512718163  \\
a_{14} &= -0.459724407978054 &\zeta_{14} &= 0.180050952622569\\
a_{16} &= -1.175355598611419 & \zeta_{16} &= 0.835173783095831 \\  
a_{18}&=0.799302873723814  & \zeta_{18} &= 0.850732311209280  \\
a_{20} &= -1.303832930713680 & \zeta_{20} &=  0.863697094160152 \\  
a_{22} &= 0.094742514998172 &\zeta_{22} &= 0.355260772731067  \\
a_{24} &= -2.293749915528502 &\zeta_{24} &= 0.463364388737488
\end{align*}
and $a_r=0$, $\zeta_r=0$ for all other $r\in\N$.


{\scriptsize
\begin{thebibliography}{10}

\bibitem{Strogatz2000}
Steven~H. Strogatz.
\newblock {From Kuramoto to Crawford: exploring the onset of synchronization in
  populations of coupled oscillators}.
\newblock {\em Physica D}, 143(1-4):1--20, 2000.

\bibitem{Pikovsky2003}
Arkady Pikovsky, Michael Rosenblum, and J{\"{u}}rgen Kurths.
\newblock {\em {Synchronization: A Universal Concept in Nonlinear Sciences}}.
\newblock Cambridge University Press, 2003.

\bibitem{Strogatz2004}
Steven~H. Strogatz.
\newblock {\em {Sync: The Emerging Science of Spontaneous Order}}.
\newblock Penguin, 2004.

\bibitem{Acebron2005}
Juan Acebr{\'{o}}n, L.~Bonilla, Conrad {P{\'{e}}rez Vicente}, F{\'{e}}lix
  Ritort, and Renato Spigler.
\newblock {The Kuramoto model: A simple paradigm for synchronization
  phenomena}.
\newblock {\em Reviews of Modern Physics}, 77(1):137--185, 2005.

\bibitem{Tchistiakov1996}
V.~Tchistiakov.
\newblock {Detecting symmetry breaking bifurcations in the system describing
  the dynamics of coupled arrays of Josephson junctions}.
\newblock {\em Physica D}, 91(1-2):67--85, 1996.

\bibitem{Ashwin2015}
Peter Ashwin, Stephen Coombes, and Rachel Nicks.
\newblock {Mathematical Frameworks for Oscillatory Network Dynamics in
  Neuroscience}.
\newblock {\em The Journal of Mathematical Neuroscience}, 6(1):2, 2016.

\bibitem{Golubitsky1988}
Martin Golubitsky, Ian Stewart, and David~G. Schaeffer.
\newblock {\em {Singularities and Groups in Bifurcation Theory. Vol. II}}.
\newblock Springer-Verlag, New York, 1988.

\bibitem{Golubitsky2002}
Martin Golubitsky and Ian Stewart.
\newblock {\em {The Symmetry Perspective}}, volume 200 of {\em Progress in
  Mathematics}.
\newblock Birkh{\"{a}}user Verlag, Basel, 2002.

\bibitem{Field2007a}
Michael~J. Field.
\newblock {\em {Dynamics and Symmetry}}, volume~3 of {\em ICP Advanced Texts in
  Mathematics}.
\newblock Imperial College Press, 2007.

\bibitem{Panaggio2015}
Mark~J. Panaggio and Daniel~M. Abrams.
\newblock {Chimera states: coexistence of coherence and incoherence in networks
  of coupled oscillators}.
\newblock {\em Nonlinearity}, 28(3):R67--R87, 2015.

\bibitem{Kuramoto2002}
Yoshiki Kuramoto and Dorjsuren Battogtokh.
\newblock {Coexistence of Coherence and Incoherence in Nonlocally Coupled Phase
  Oscillators}.
\newblock {\em Nonlinear Phenomena in Complex Systems}, 5(4):380--385, 2002.

\bibitem{Abrams2004}
Daniel~M. Abrams and Steven~H. Strogatz.
\newblock {Chimera States for Coupled Oscillators}.
\newblock {\em Physical Review Letters}, 93(17):174102, 2004.

\bibitem{Abrams2008}
Daniel~M. Abrams, Renato~E. Mirollo, Steven~H. Strogatz, and Daniel~A. Wiley.
\newblock {Solvable Model for Chimera States of Coupled Oscillators}.
\newblock {\em Physical Review Letters}, 101(8):084103, 2008.

\bibitem{Omel'chenko2013}
Oleh~E. Omel'chenko.
\newblock {Coherence-incoherence patterns in a ring of non-locally coupled
  phase oscillators}.
\newblock {\em Nonlinearity}, 26(9):2469--2498, 2013.

\bibitem{Ashwin2014a}
Peter Ashwin and Oleksandr Burylko.
\newblock {Weak chimeras in minimal networks of coupled phase oscillators}.
\newblock {\em Chaos}, 25:013106, 2015.

\bibitem{Bick2015c}
Christian Bick and Peter Ashwin.
\newblock {Chaotic weak chimeras and their persistence in coupled populations
  of phase oscillators}.
\newblock {\em Nonlinearity}, 29(5):1468--1486, 2016.

\bibitem{Sethia2013}
Gautam~C. Sethia, Abhijit Sen, and George~L. Johnston.
\newblock {Amplitude-mediated chimera states}.
\newblock {\em Physical Review E}, 88(4):042917, 2013.

\bibitem{Zakharova2014}
Anna Zakharova, Marie Kapeller, and Eckehard Sch{\"{o}}ll.
\newblock {Chimera death: symmetry breaking in dynamical networks.}
\newblock {\em Physical Review Letters}, 112(15):154101, 2014.

\bibitem{Panaggio2015b}
Mark~J. Panaggio, Daniel~M. Abrams, Peter Ashwin, and Carlo~R. Laing.
\newblock {Chimera states in networks of phase oscillators: The case of two
  small populations}.
\newblock {\em Physical Review E}, 93(1):012218, 2016.

\bibitem{Schwartzman1957}
Sol Schwartzman.
\newblock {Asymptotic Cycles}.
\newblock {\em The Annals of Mathematics}, 66(2):270--284, 1957.

\bibitem{Oxtoby1952}
John~C. Oxtoby.
\newblock {Ergodic Sets}.
\newblock {\em Bulletin of the American Mathematical Society}, 58(2):116--137,
  1952.

\bibitem{Ashwin1995}
Peter Ashwin.
\newblock {Attractors stuck on to invariant subspaces}.
\newblock {\em Physics Letters A}, 209(5-6):338--344, 1995.

\bibitem{Melbourne1993}
Ian Melbourne, Michael Dellnitz, and Martin Golubitsky.
\newblock {The structure of symmetric attractors}.
\newblock {\em Archive for Rational Mechanics and Analysis}, 123(1):75--98,
  1993.

\bibitem{Katok1995}
Anatole Katok and Boris Hasselblatt.
\newblock {\em {Introduction to the Modern Theory of Dynamical Systems}},
  volume~54 of {\em Encyclopedia of Mathematics and its Applications}.
\newblock Cambridge University Press, Cambridge, 1995.

\bibitem{Young2002}
Lai-Sang Young.
\newblock {What Are SRB Measures, and Which Dynamical Systems Have Them?}
\newblock {\em Journal of Statistical Physics}, 108(5/6):733--754, 2002.

\bibitem{Barany1993}
Ernest Barany, Michael Dellnitz, and Martin Golubitsky.
\newblock {Detecting the symmetry of attractors}.
\newblock {\em Physica D}, 67(1-3):66--87, 1993.

\bibitem{Dellnitz1994}
Michael Dellnitz, Martin Golubitsky, and Matthew Nicol.
\newblock {Symmetry of Attractors and the Karhunen-Lo{\`{e}}ve Decomposition}.
\newblock In Lawrence Sirovich, editor, {\em Trends and Perspectives in Applied
  Mathematics}, chapter~4, pages 73--108. Springer New York, 1994.

\bibitem{Ashwin1997}
Peter Ashwin and Matthew Nicol.
\newblock {Detection of symmetry of attractors from observations I. Theory}.
\newblock {\em Physica D}, 100(1):58--70, 1997.

\bibitem{Walsh1995}
Jams~A. Walsh.
\newblock {Rotation Vectors for Toral Maps and Flows: A Tutorial}.
\newblock {\em International Journal of Bifurcation and Chaos},
  05(02):321--348, 1995.

\bibitem{Ashwin1992}
Peter Ashwin and James~W. Swift.
\newblock {The dynamics of n weakly coupled identical oscillators}.
\newblock {\em Journal of Nonlinear Science}, 2(1):69--108, 1992.

\bibitem{Ashwin2015a}
Peter Ashwin and Ana Rodrigues.
\newblock {Hopf normal form with S{\_}N symmetry and reduction to systems of
  nonlinearly coupled phase oscillators}.
\newblock {\em Physica D}, 325:14--24, 2016.

\bibitem{Golubitsky2006b}
Martin Golubitsky, Kresimir Josic, and Eric Shea-Brown.
\newblock {Winding Numbers and Average Frequencies in Phase Oscillator
  Networks}.
\newblock {\em Journal of Nonlinear Science}, 16(3):201--231, 2006.

\bibitem{Antoneli2006}
Fernando Antoneli and Ian Stewart.
\newblock {Symmetry and Synchrony in Coupled Cell Networks 1: Fixed-Point
  Spaces}.
\newblock {\em International Journal of Bifurcation and Chaos},
  16(03):559--577, 2006.

\bibitem{Bick2016b}
Christian Bick, Peter Ashwin, and Ana Rodrigues.
\newblock {Chaos in generically coupled phase oscillator networks with
  nonpairwise interactions}.
\newblock {\em Chaos}, 26(9):094814, 2016.

\bibitem{Misiurewicz1989}
Micha{\l} Misiurewicz and Krystyna Ziemian.
\newblock {Rotation Sets for Maps of Tori}.
\newblock {\em Journal of the London Mathematical Society}, s2-40(3):490--506,
  1989.

\bibitem{Ashwin2016}
Peter Ashwin, Christian Bick, and Oleksandr Burylko.
\newblock {Identical Phase Oscillator Networks: Bifurcations, Symmetry and
  Reversibility for Generalized Coupling}.
\newblock {\em Frontiers in Applied Mathematics and Statistics}, 2(7), 2016.

\bibitem{Dionne1996a}
Benoit Dionne, Martin Golubitsky, and Ian Stewart.
\newblock {Coupled cells with internal symmetry: I. Wreath products}.
\newblock {\em Nonlinearity}, 9(2):559--574, 1996.

\bibitem{Bick2011}
Christian Bick, Marc Timme, Danilo Paulikat, Dirk Rathlev, and Peter Ashwin.
\newblock {Chaos in Symmetric Phase Oscillator Networks}.
\newblock {\em Physical Review Letters}, 107(24):244101, 2011.

\bibitem{Bick2012b}
Christian Bick.
\newblock {\em {Chaos and Chaos Control in Network Dynamical Systems}}.
\newblock Ph.D. dissertation, Georg-August-Universit{\"{a}}t G{\"{o}}ttingen,
  2012.

\bibitem{Chossat1988}
Pascal Chossat and Martin Golubitsky.
\newblock {Symmetry-increasing bifurcation of chaotic attractors}.
\newblock {\em Physica D}, 32(3):423--436, 1988.

\bibitem{Fried1982}
David Fried.
\newblock {The geometry of cross sections to flows}.
\newblock {\em Topology}, 21(4):353--371, 1982.

\bibitem{Jenkinson2006}
Oliver Jenkinson.
\newblock {Ergodic Optimization}.
\newblock {\em Discrete and Continuous Dynamical Systems}, 15(1):197--224,
  2006.

\bibitem{Ashwin2006}
Peter Ashwin, Oleksandr Burylko, Yuri~L. Maistrenko, and Oleksandr~V. Popovych.
\newblock {Extreme Sensitivity to Detuning for Globally Coupled Phase
  Oscillators}.
\newblock {\em Physical Review Letters}, 96(5):054102, 2006.

\bibitem{Omelchenko2010}
Oleh~E. Omel'chenko, Matthias Wolfrum, and Yuri~L. Maistrenko.
\newblock {Chimera states as chaotic spatiotemporal patterns}.
\newblock {\em Physical Review E}, 81(6):3--6, 2010.

\bibitem{Wolfrum2011b}
Matthias Wolfrum and Oleh~E. Omel'chenko.
\newblock {Chimera states are chaotic transients}.
\newblock {\em Physical Review E}, 84(1):015201, 2011.

\bibitem{Omelchenko2015}
Oleh~E. Omel'chenko.
\newblock {Private Communication}, 2015.

\end{thebibliography}
}
\end{document}